\documentclass[11pt]{amsart}
\usepackage{amssymb, latexsym, mathrsfs,  color, mathtools}
\usepackage[colorlinks=true, pdfstartview=FitV,linkcolor=blue,citecolor=blue,urlcolor=blue]{hyperref}

    \advance \topmargin by -\headheight
    \advance \topmargin by -\headsep

    \evensidemargin \oddsidemargin
\newtheorem {theorem}    {Theorem}[section]

\newtheorem {lemma}      [theorem]    {Lemma}
\newtheorem {corollary}  [theorem]    {Corollary}
\newtheorem {proposition}[theorem]    {Proposition}

\theoremstyle{definition}
\newtheorem{definition}[theorem]{Definition}
\newtheorem{remark}[theorem]{Remark}

\numberwithin{equation}{section}

\usepackage[colorinlistoftodos]{todonotes}

\newenvironment{red}{\relax\color{red}}{\relax}

\newcommand{\ber}{\begin{red}}
\newcommand{\er}{\end{red}}

\begin{document}

\title[Convergence of Kac--Moody Eisenstein Series]{Convergence of Kac--Moody Eisenstein Series \\ over a function field}

\date{\today}

\author[K.-H. Lee]{Kyu-Hwan Lee$^{\star}$}
\thanks{$^{\star}$This work was partially supported by a grant from the Simons Foundation (\#712100).}
\address{Department of
Mathematics, University of Connecticut, Storrs, CT 06269, U.S.A.}
\email{khlee@math.uconn.edu}

\author[D. Liu]{Dongwen Liu$^{\dagger}$}
\thanks{$^{\dagger}$This work was partially supported by the Natural Science Foundation of Zhejiang Province (\#LZ22A010006) and the National Natural Science Foundation of 
China (\#12171421).} 
\address{School of Mathematical Sciences, Zhejiang University, Hangzhou 310027, P.R. China}
\email{maliu@zju.edu.cn}

\author[T. Oliver]{Thomas Oliver}
\address{Teesside University, Middlesbrough, U.K.}
\email{T.Oliver@tees.ac.uk}


\begin{abstract}  
We establish everywhere convergence in a natural domain for Eisenstein series on a symmetrizable Kac--Moody group over a function field. Our method is different from that of the affine case which does not directly generalize. In comparison with the analogous result over the real numbers, everywhere convergence is achieved without any additional condition on the root system.
\end{abstract}

\maketitle

\section{Introduction}\label{sec:intro}

Eisenstein series are amongst the most fundamental examples of automorphic forms. 
They are of significant consequence in number theory, and have applications in other areas of mathematics and mathematical physics. 
The theory of Eisenstein series on reductive groups was developed by Langlands in \cite{La1, La2} and led ultimately to the Principle of Functoriality. 
Furthermore, as surveyed in \cite[Section~8]{GM}, Eisenstein series form the basis of an established technique for proving cases of functoriality, namely, the Langlands--Shahidi method. 

Motivated by potential arithmetic applications such as generalizing the Langlands--Shahidi method, we study Eisenstein series over Kac--Moody groups. In a series of pioneering works, affine Kac--Moody Eisenstein series over $\mathbb{R}$ were extensively studied by Garland \cite{G78,G80,G99,G04,G06,GMS1,GMS2,GMS3,GMS4,G11}. 
Some of the results in the affine case were generalized to number fields by Liu \cite{Liu}, rank $2$ hyperbolic Eisenstein series over $\mathbb R$ by Carbone--Lee--Liu \cite{CLL}, 
and general symmetrizable Kac--Moody Eisenstein series over $\mathbb{R}$ by Carbone--Garland--Lee--Liu--Miller \cite{CGLLM}. 
Entirety of certain cuspidal Kac--Moody Eisenstein series over $\mathbb R$ is established for the affine case in \cite{GMP}, for the rank 2 hyperbolic case in \cite{CLL}, and for more general cases in \cite{CLL1}.
Applications of Kac--Moody Eisenstein series to string theory appeared in \cite{FK,FKP}.

Classical Eisenstein series over function fields were studied by Harder \cite{Ha}. 
Important results for affine Kac--Moody Eisenstein series over function fields were announced by Braverman--Kazhdan \cite{BK}, utilising on a geometric formalism introduced by Kapranov \cite{Kap} and Patnaik \cite{P}. 
A more algebraic framework was established by Lee--Lombardo in the affine case \cite{LL}. In this paper we are interested in general Kac--Moody Eisenstein series over function fields.
 
As for general Eisenstein series over symmetrizable Kac--Moody groups, it has been a challenging problem to establish everywhere convergence of the series on the natural domain, while almost everywhere convergence of the series follows from the convergence of their constant terms. 
Over $\mathbb R$, the main result of \cite{CGLLM} establishes everywhere convergence under some restrictions. 
Over function fields, everywhere convergence of the affine Kac--Moody Eisenstein series follows from that of its constant term through an explicit description of the unipotent subgroup as was shown in \cite{LL}. 
However, such an explicit description is not available for general Kac--Moody groups, and the method in \cite{LL} does not directly generalize. 

In this paper, we circumvent the obstacle and establish everywhere convergence of the Kac--Moody Eisenstein series over a function field on a natural domain, i.e., for the Godement range of the spectral parameter and for the Tits cone range of the group element. 
Our approach involves constructing a subset of positive measure inside the unipotent group through various representation-theoretic properties.
It requires convergence of the constant terms as a prerequisite which we obtain through the similar argument in \cite{CGLLM}. 

In order to state our main result more precisely, we need to introduce some notation.
Let $G$ be a symmetrizable Kac--Moody group over a function field $F$ with the Iwasawa decomposition $G_{\mathbb A} = U_{\mathbb A} H_{\mathbb A} \mathbb K$, where $\mathbb A$ is the ad\`ele ring of $F$, $U$ is a pro-unipotent subgroup, $H$ is a torus, and $\mathbb K$ is an analogue of a maximal compact subgroup.  
Let $\mathfrak g$ be the Kac--Moody algebra of $G$, and $\mathfrak h$ be its Cartan subalgebra. 
Denote by $\rho \in \mathfrak h^*_{\mathbb C}$ the Weyl vector and by $\mathcal C^* \subset \mathfrak h^*_{\mathbb R}$ the set of strictly dominant weights. 
Finally, let $H_\mathfrak{C} \subset H_{\mathbb A}$ be the subgroup corresponding the Tits cone $\mathfrak C \subset \mathfrak h_{\mathbb R}$.
The main result of this paper can be stated as follows:

\begin{theorem}\label{conv.thm.intro}
If $\lambda\in\mathfrak{h}^{\ast}_{\mathbb{C}}$ satisfies $Re(\lambda - \rho) \in \mathcal{C}^*$, then the Kac--Moody Eisenstein series $E_\lambda(g)$ converges absolutely for $g\in U_{\mathbb{A}}H_\mathfrak{C}\mathbb{K}$.
\end{theorem}

In comparison with the real case considered in \cite{CGLLM}, we do not impose conditions on the root system of $G$ .

We conclude the introduction with an outline of what follows. 
In Section~\ref{Zform}, we define an adelic Kac--Moody group $G_{\mathbb{A}}$ over a function field and recall its Iwasawa decomposition (Proposition~\ref{thm:AdelicIwasawa}).
In Section~\ref{constant-term}, we define Eisenstein series on $G_{\mathbb{A}}$ and establish convergence of their constant terms (Proposition~\ref{pro-conv-const}). In Section~\ref{Conv}, we prove the main result Theorem \ref{conv.thm.intro} after establishing a crucial lemma (Lemma \ref{lem-pm'}) .

\subsection*{Acknowledgments}
We are grateful to Lisa Carbone, Howard Garland, Stephen D. Miller and Manish Patnaik for helpful discussions.

\section{Preliminaries} \label{Zform}

In this section, we fix notation and recall the Iwasawa decomposition.

Let $A=(a_{ij})_{i,j\in I}$ be a non-singular\footnote{This condition excludes the affine case which is well understood.} symmetrizable generalized Cartan matrix, indexed by $I=\{1,\dots,n\}$, let $\left(\mathfrak{h}_{\mathbb{C}},\Delta,\Delta^{\vee}\right)$ be a realisation of $A$ with $\Delta=\{ \alpha_1, ... , \alpha_n\} \subset \mathfrak h_{\mathbb{C}}^*$ (resp. $\Delta^{\vee}=\{ \alpha^\vee_1, ... , \alpha^\vee_n \} \subset \mathfrak h_{\mathbb{C}}$), and let $\mathfrak g_{\mathbb C}$ be the associated {complex} Kac--Moody algebra.
Since $A$ is non-singular, we have $\mathfrak{h}_{\mathbb{C}}=\mathrm{Span}_{\mathbb{C}}\{ \alpha_1^\vee, ... , \alpha_n^\vee\} $ (resp. $\mathfrak{h}^{\ast}_{\mathbb{C}}=\mathrm{Span}_{\mathbb{C}}\{ \alpha_1, ... , \alpha_n\} $).
We denote by $ Q\subseteq \mathfrak{h}_{\mathbb{C}}^*$ (resp. $ Q^{\vee}\subseteq\mathfrak{h}_{\mathbb{C}}$) the integral linear span of $\Delta$ (resp. $\Delta^{\vee}$).

Let $\Phi$ be the set of roots of $\mathfrak{g}_{\mathbb{C}}$, let $\Phi_{+}$ (resp. $\Phi_{-}$) be the set of positive (resp. negative) roots with respect to $\Delta$, and let $W$ be the Weyl group of $\mathfrak g_{\mathbb C}$ generated by the simple reflections $w_i$ ($i\in I$) corresponding to the simple roots $\alpha_i$.
A root $\alpha \in \Phi$ is called a {\em real} root if there exists $w \in W$
such that $w\alpha$ is a simple root. A root $\alpha$ which is not real is referred to as {\em imaginary}.
For each real root $\alpha$, written as $w\alpha_i$
for some $w \in W$ and $i \in I$, its associated coroot is well-defined by the formula $\alpha^\vee = w \alpha_i^\vee$.
For $w\in W$, we define 
\begin{equation}\label{eq.Phiw}
\Phi_w=\Phi_{+}\cap w^{-1}\Phi_{-}.
\end{equation}
Denote by $e_i$ and $f_i$, $i \in I$, the Chevalley generators of $\mathfrak{g}_{\mathbb{C}}$.
Let ${\mathcal U}_{\mathbb{C}}$ be the universal enveloping algebra of $\mathfrak{g}_{\mathbb C}$,
and let ${\mathcal U}_{{\mathbb{Z}}}\subseteq {\mathcal U}_{{\mathbb{C}}}$ be the ${\mathbb{Z}}$--subalgebra generated by the set:
\[
\left\{\dfrac{e_i^{m}}{m!},\ \dfrac{f_i^{m}}{m!},\ \left (\begin{matrix}
h\\ m\end{matrix}\right ):i\in I,h\in Q^{\vee},m\in\mathbb{Z}_{\geq 0}\right\},
\] where  $\left (\begin{matrix} h\\ m\end{matrix}\right ) \coloneqq \frac{h(h-1)\cdots(h-m+1)}{m!}$.
Fix a nonzero 
dominant integral weight $\Lambda \in \mathfrak h_{\mathbb C}^*$,
let $\left(V,\pi\right)$
be the corresponding irreducible highest weight representation of $\mathfrak{g}_{\mathbb C}$,
and let $v\in V$ be a non-zero highest weight vector.
We set $V_{\mathbb{Z}}\ =\ \mathcal{U}_{\mathbb{Z}}\cdot v$ and, for a field $F$, we set $V_{F}=F\otimes_{\mathbb{Z}}V_{\mathbb{Z}}$. 

For all $i\in I$, the Chevalley generators $e_i,f_i$ are locally nilpotent on $V_{F}$.
Given $i\in I$ and $r,s\in F$, we introduce:
\begin{equation}\label{eq.ualphai}
u_{\alpha_i}(r)=\exp(\pi(re_i)),\ \ u_{-\alpha_i}(s)=\exp(\pi(sf_i))\in\mathrm{Aut}\left(V_F\right).
\end{equation}
We define $G^0_{F}$ to be the subgroup of $\mathrm{Aut}\left(V_{F}\right)$ generated by $u_{\alpha_i}(r)$ and $u_{-\alpha_i}(s)$, with $r,s$ varying in $F$ and $i$ varying in $I$.

Choose a coherently ordered basis $\mathfrak{X}=\{v_1,v_2,\dots\}$ of $V_{\mathbb Z}$, as defined in \cite[Section~5]{CG}.
For $t\in\mathbb{Z}_{>0}$, let $A_t\subset G^0_F$ be the following subset:
\[
A_t=\left\{ g \in G^0_F : g v_i =v_i, \ i=1, 2, \dots , t \right\}.
\]
The sets $A_t$, indexed by $t\in\mathbb{Z}_{>0}$, form a base of open neighborhoods of the identity to define a topology $\tau$ on $G^0_F$. 
\begin{definition}\label{def.KMgroup}
The {\em Kac--Moody group} $G_F$ is the completion
of $G^0_F$ with respect to $\tau$.
\end{definition}
The group $G_F$ in Definition~\ref{def.KMgroup} is usually referred to as the representation theoretic construction of a Kac--Moody group. 
One may to define a Kac--Moody group through other means, for example, see \cite{Ku, T}.

Let $B^0_F$ be the subgroup of $G^0_F$ consisting of the elements represented by upper triangular matrices with respect to $\mathfrak{X}$.
For $i\in I$ and $s\in F^{\times}$, we set:
\begin{equation}\label{eq.walphai}
w_{i}(s) = u_{\alpha_i}(s) u_{-\alpha_i}(-s^{-1}) u_{\alpha_i}(s)\in G^0_F,
\end{equation}
where $u_{\alpha_i}$ is as in equation~\eqref{eq.ualphai}. 
For a real root $\alpha$, choose once for all an expression  
  $\alpha=w_{i_1}\cdots w_{i_\ell}\alpha_i$ for some $i\in I$, and define the
corresponding one-parameter subgroup 
\[
  u_\alpha(s) = w u_{\alpha_i}( s)w^{-1} \in \textrm{Aut}(V_{F}) \qquad (s \in F),
\]
where   $w =w_{i_1}(1)\cdots w_{i_\ell}(1)$.

For $t \in \mathbb Z_{>0}$ we let $V_t$ be the span of the $v_m \in \mathfrak{X}$ for $m \le t$, so that $B^0_FV_t \subseteq V_t$ for each $t$.
Let $B_{F,t}$ be the image of $B^0_F$ in $\mathrm{Aut}(V_t)$.
For $t'\geq t$, we have surjective homomorphisms $\pi_{tt'}: B_{F,t'} \longrightarrow B_{F,t}$.
We define $B_F\leq G_F$ to be the projective limit of the projective family $\{ B_{F,t}, \pi_{tt'} \}$.

Let $N_F$ be the subgroup of $G_F$ generated by $w_{i}(s)$, $i \in I$, $s \in F^\times$, and let $U_F\subset B_F$ be the subgroup of elements acting as unipotent upper triangular matrices with respect to $\mathfrak{X}$. 
We define, for $i\in I$ and $s\in F^{\times}$,
\begin{equation}\label{eq.hi}
h_{i}(s)=w_{i}(s)w_{i}(1)^{-1}.
\end{equation}
We denote by $H_F\subset G_F$  the group generated by $h_i(s)$ as $i$ varies over $I$ and $s$ varies over $F^{\times}$.

An important fact about the structure of $G_F$ is established by the following proposition.
\begin{proposition}[{\cite[Theorem~6.1]{CG}}]\label{prop.BNpair}
The pair $(B_F, N_F)$ is a $BN$-pair for the group $G_F$.
\end{proposition}

We identify the Weyl group $W$ with the group $N_F/(B_F \cap N_F)$ in such a way that $w_i$ is represented by $w_i(1)$. From a standard property of a $BN$-pair we obtain

\begin{corollary}\label{cor:Bruhat}
The group $G_F$ admits a Bruhat decomposition, that is: 
\begin{equation}\label{eq.Bruhat}
G_F= \bigcup_{w \in W} B_FwB_F,
\end{equation}
in which the union is disjoint.
\end{corollary}

From now on, let $F$ be the function field of a smooth projective curve $X$ defined over a finite field $\mathbb{F}_q$.
For a place $v$ of $F$, we denote the corresponding completion by $F_v$. 
Denoting by $\mathcal{O}_v$ the ring of integers in $F_v$, we define $G^0_{\mathcal{O}_v}$ to be the subgroup of $\mathrm{Aut}\left(V_{F_v}\right)$ generated by $u_{\alpha_i}(r)$ and $u_{-\alpha_i}(s)$, with $r,s$ varying in $\mathcal{O}_v$ and $i$ varying in $I$, and define $G_{\mathcal{O}_v}$ to be its completion.
For convenience, we will write:
\[ G_v = G_{F_v}, \quad U_v = U_{F_v}, \quad H_v = H_{F_v}, \quad K_v=G_{\mathcal O_v}. \]
A proof of the following proposition is essentially the same as the affine case, and we do not reproduce it here.
\begin{proposition}[\cite{G80}]\label{iwasawa}
For each place $v$ of the field $F$, the group $G_v$ admits an Iwasawa decomposition. That is:
\begin{equation}\label{eq.LocalIwasawa}
G_{v}= U_{v}H_vK_v.
\end{equation}
\end{proposition}
In particular, we may write each $g_v\in G_v$ as a product $g_v=u_vh_vk_v$ with $u_v\in U_v$, $h_v\in H_v$, and $k_v\in K_v$. 
We note that such an expression for $g_v$ is not unique.
{For a choice of expression $g_v=u_vh_vk_v$, we will use the notation $\mathrm{Iw}_{H_v}(g)$ to denote $h_v$.}

Let $\mathcal V$ be the set of places of $F$.
The ad\`{e}le ring $\mathbb A$ of $F$ is the restricted direct product over $\mathcal V$ of $F_v$ with respect to the subrings $\mathcal O_{v}$.
We introduce the following subgroups:
\begin{equation}\label{eq.adelicgroups}
{U}_{\mathbb A} = {\prod_{v \in \mathcal{V}}}' \, {U}_v,
\quad
{H}_{\mathbb{A}} = {\prod_{v \in \mathcal{V}}}'\ {H}_{v},
\quad
{\mathbb K}=\prod_{v \in \mathcal V}{K}_v ,
\end{equation}
in which the first product is restricted with respect to ${U}_{v} \cap {K}_v$, the second product is restricted with respect to ${H}_v\cap{K}_v$, and the third product is unrestricted.
For later use, we also define $B_{\mathbb{A}}\subset G_{\mathbb{A}}$ to be the restricted direct product $\prod'_{v\in\mathcal{V}}B_v$ with respect to $B_v\bigcap K_v$. 

Taking a product of equation~\eqref{eq.LocalIwasawa} over $v\in\mathcal{V}$, we deduce the following proposition.

\begin{proposition}\label{thm:AdelicIwasawa}
The group $G_{\mathbb{A}}$ admits an Iwasawa decomposition, that is:
\begin{equation}\label{eq.adelicIwasawa}
G_{\mathbb A} \, = \, {U}_{\mathbb A} H_{\mathbb A} {\mathbb K}.
\end{equation}
\end{proposition}

As in the local case, an expression $g=uhk$ as per equation~\eqref{eq.adelicIwasawa} is not uniquely determined. For $g\in G_{\mathbb{A}}$ and a choice of expression $g=uhk$, we will use the notation \begin{equation} \label{hgh} \mathrm{Iw}_H(g)=h. \end{equation} 

\section{Kac--Moody Eisenstein series}\label{constant-term}

In this section the Kac--Moody Eisenstein series will be defined, and convergence of its constant term will be established. As a result, we will obtain almost everywhere convergence of the Eisenstein series. 

We maintain the notations from Section \ref{Zform}. 
Furthermore we let $\langle~,~\rangle$ denote the natural pairing between $\mathfrak{h}_{\mathbb{C}}$ and $\mathfrak{h}_{\mathbb{C}}^{\ast}$, 
so that $\langle\alpha_j, \alpha^\vee_i\rangle = a_{ij}$ is the $(i,j)$-entry of the generalized Cartan matrix $A$. 
Recall that the fundamental weights form the basis of $\mathfrak{h}_{\mathbb{C}}^{\ast}$ dual to $\Delta^{\vee}\subset\mathfrak{h}_{\mathbb{C}}$. 
We denote the integral span of the fundamental weights by $P\subset\mathfrak{h}_{\mathbb{C}}^{\ast}$.
Since $A$ is non-singular, there exists a unique 
$\rho\in\mathfrak{h}_{\mathbb{C}}^*$ such that, for all $i\in I$, we have $\langle\rho,\alpha_i^\vee\rangle =1$. 
Set $\mathfrak{h}_{\mathbb{R}}=\mathrm{Span}_{\mathbb{R}}\{\alpha^{\vee}_1,\dots,\alpha^{\vee}_r\}$ (resp. $\mathfrak{h}^{\ast}_{\mathbb{R}}=\mathrm{Span}_{\mathbb{R}}\{\alpha_1,\dots,\alpha_r\}$).

Recall that $F$ is the function field of a smooth projective curve $X$ defined over a finite field $\mathbb{F}_q$ and that, for a place $v$ of $F$, we denote the corresponding completion by $F_v$. 
For clarity, we use the notation $h_{i,v}$ for $h_i$ in equation~\eqref{eq.hi} when it is applied to an element of $F_v^\times$. 
By construction we may write each $h\in H_{\mathbb{A}}$ as $h=(h_v)_{v\in\mathcal{V}}$, where $h_v=\prod_{i\in I}h_{i,v}(t_{i,v})\in H_v$ and $t_{i,v}\in F_v^{\times}$. 
For $v\in \mathcal{V}$, we denote by $|\cdot|_v$ the normalized absolute value on $F_v$. 
For $\lambda\in\mathfrak{h}_{\mathbb{C}}^{\ast}$, we write
\begin{equation}\label{eq.hlambdalocal}
h^{\lambda}=\prod_{i\in I}\prod_{v\in\mathcal{V}}|t_{i,v}|_v^{\left\langle\lambda,\alpha_i^\vee \right\rangle}.
\end{equation}
By definition, for all but finitely many $v\in\mathcal{V}$, we have $h_v\in H_v \cap K_v$.

\begin{lemma} \label{lem-all-but}
For $v\in\mathcal{V}$, if $h_v=\prod_{i \in I} h_{i,v}(t_{i,v}) \in H_v \cap K_v$ then $t_{i,v} \in \mathcal{O}_v{^\times}$ for all $i \in I$. 
\end{lemma}

\begin{proof}
For $i\in I$, denote the corresponding fundamental weight by $\Lambda_i$. Recall that we fixed a regular dominant integral weight $\Lambda = \sum_{i \in I} m_i \Lambda_i$ $(m_i \in \mathbb Z_{>0})$ to construct $G_v$ is in Section~\ref{Zform}. Choose a highest weight vector $v_0\in V_{\mathbb{Z}}$ (with weight $\Lambda$) and consider $1\otimes v_0\in V_{F_v}=F_v\otimes V_{\mathbb{Z}}$. 
Then we have:
\[h_v(1\otimes v_0)=\prod_{i \in I} t_{i,v}^{m_i}\otimes v_0.\]
Since $K_v$ preserves $V_{\mathcal{O}_v}$, we deduce that $\sum_{i \in I} m_i\,  \mathrm{ord}_v(t_{i,v})\geq0$, where $\mathrm{ord}_v$ is the $v$-adic valuation on $F_v$.
On the other hand, since $H_v\cap K_v$ is a group, we have $h_v^{-1}=\prod_{i\in I}h_{\alpha_i}(t_{i,v}^{-1})\in H_v\cap K_v$.
By the same argument, we deduce that $\sum_{i \in I} m_i \, \mathrm{ord}_v(t_{i,v})\leq0$. Consequently, we have
\begin{equation} \label{suiI} \sum_{i \in I} m_i \, \mathrm{ord}_v(t_{i,v})=0. \end{equation}
Since $\Lambda$ is regular, i.e., $m_j >0$ for each $j$, we can choose a weigh vector $v_j \in V_{\mathbb Z}$ with weight $\Lambda - \alpha_j$ for each $j \in I$.
By applying $h_v$ and $h_v^{-1}$ to $v_j$, we obtain
\begin{equation} \label{suiIJ} \sum_{i \in I} (m_i-a_{ij}) \, \mathrm{ord}_v(t_{i,v}) =0, \end{equation} where $A=(a_{ij})$ is the generalized Cartan matrix.
Combining \eqref{suiI} and \eqref{suiIJ}, we have \[ \sum_{i \in I} a_{ij} \mathrm{ord}_v (t_{i,v}) =0, \qquad \text{ for each } j . \] Since $A$ is nonsingular, we have $\mathrm{ord}_v (t_{i,v}) =0$ and $t_{i,v} \in \mathcal O^\times$ for all $i \in I$, as desired. 
\end{proof}

By Lemma \ref{lem-all-but}, the infinite product in equation~\eqref{eq.hlambdalocal} converges. Denoting the ad\`ele $(t_{i,v})_{v\in\mathcal{V}}\in \mathbb{A}$ by $t_i$ and the adelic norm by $|\cdot|$, we have 
\begin{equation}\label{eq.hlambdaadelic}
h^{\lambda}=\prod_i|t_i|^{\langle\lambda,\alpha_i^\vee\rangle}, \qquad \lambda \in \mathfrak h_{\mathbb C}^*.
\end{equation}

\begin{lemma}
The function $g\mapsto \mathrm{Iw}_H(g)^{\lambda}$ is  well-defined on $G_{\mathbb{A}}$ for $ \lambda \in \mathfrak h_{\mathbb C}^*$, where the notation $\mathrm{Iw}_H(g)$ is introduced in \eqref{hgh}. 
\end{lemma}

\begin{proof}
Assume that $g=u_1h_1k_1=u_2h_2k_2$ with respect to the Iwasawa decomposition. Rearranging $u_1h_1k_1=u_2h_2k_2$, we get
\[k_1k_2^{-1}=(u_1h_1)^{-1}u_2h_2\in B_{\mathbb{A}}\cap\mathbb{K}.\]
Since $B_{\mathbb{A}}\cap \mathbb{K}=(U_{\mathbb{A}}\cap\mathbb{K})\rtimes(H_{\mathbb{A}}\cap\mathbb{K})$, we may write $k_1k_2^{-1}=u_3h_3$ for some $u_3\in U_{\mathbb{A}}\cap\mathbb{K}$ and $h_3\in H_{\mathbb{A}}\cap\mathbb{K}$. Therefore, since  $H_{\mathbb{A}}$ normalizes $U_{\mathbb{A}}$, we may write
\[u_2h_2k_2=u_1h_1k_1=u_1h_1u_3h_3k_2 = u_1u_4h_1h_3 k_2\]
for some $u_4\in U_{\mathbb{A}}$, and obtain $h_2=h_1h_3$ from $H_\mathbb{A} \cap U_{\mathbb A}$ being trivial.  To conclude, note that 
\[\mathrm{Iw}_H(u_2h_2k_2)^\lambda=h_2^{\lambda}=(h_1h_3)^{\lambda}=h_1^{\lambda}h_3^{\lambda}=h_1^{\lambda}=\mathrm{Iw}_H(u_1h_1k_1)^\lambda .\]
\end{proof}

For a place $v$ of $F$, the natural inclusion $F\rightarrow F_v$ induces a map $\iota_v:G_F\rightarrow G_v.$  Therefore, we get a map $\iota=\prod\iota_v:G_F\rightarrow\prod G_{v}.$ Define $\Gamma=\{\gamma\in G_F:\iota(\gamma)\in G_{\mathbb{A}}\}$. Whenever appropriate, elements of  $\Gamma$ will be identified with their images in $G_{\mathbb A}$ via $\iota$.

\begin{remark}
As the definition of $\Gamma$ suggests, it is not always true that the image of $\iota$ is contained in $G_{\mathbb{A}}$. In the affine case, it is discussed, for example, in \cite[Example 3.13]{LL}.  
\end{remark}

Now we define the Kac--Moody Eisenstein series.
\begin{definition}\label{Eisenstein.definition}
Given $\lambda\in \mathfrak h_{\mathbb{C}}^*$, the {\em Eisenstein series} $E_{\lambda}$ is defined to be:
\begin{equation}\label{eq:Eisenstein}
E_{\lambda}(g)\quad = \sum_{\gamma\in \Gamma\cap{B_{\mathbb{A}}}\backslash \Gamma}
\mathrm{Iw}_H(\gamma g)^{\lambda+\rho}, \qquad g \in G_{\mathbb{A}}.
\end{equation}
\end{definition}

Note that we may regard $E_{\lambda}$ as a function on $(\Gamma\cap U_{\mathbb{A}}\backslash U_{\mathbb{A}}) \times H_{\mathbb{A}}$.
In the rest of this section we will define the constant term of $E_{\lambda}$ and establish convergence of the constant term. The result will be crucial in proving the main theorem. 

As in \cite{G04, LL}, the space $\Gamma\cap U_{\mathbb{A}} \backslash U_{\mathbb{A}}$ admits a projective limit measure $du$, which is a $U_{\mathbb{A}}$-invariant probability measure.

\begin{definition}\label{def:const}
Given $\lambda\in \mathfrak h_{\mathbb{C}}^*$, the {\em constant term} of $E_\lambda$ is defined to be:
\begin{equation}\label{eq.constdef}
E^\sharp_\lambda(g)=\int_{\Gamma\cap U_{\mathbb{A}}\backslash U_{\mathbb{A}}}E_\lambda(ug)du, \qquad g \in G_{\mathbb{A}}.
\end{equation}
\end{definition}

\begin{remark}
In the proof of Proposition \ref{pro-conv-const}, we only need $\lambda \in \mathfrak h_{\mathbb R}^*$. In this case, we interpret the infinite sum $E_\lambda$ as a function taking values in $\mathbb{R}_+\cup\{\infty\}$, and $E^\sharp_\lambda$ is well-defined. 
After convergence is established for $\lambda \in \mathfrak h_{\mathbb R}^*$, the above definition of $E_\lambda^{\sharp}$ will be valid for $\lambda \in \mathfrak h_{\mathbb C}^*$ by dominance. 
\end{remark}

Applying the Gindikin--Karpelevich formula,  a formal calculation as in \cite{La1, G04} 
yields 
\begin{equation} \label{eee}
E^\sharp_\lambda(g)=\sum_{w\in W}\mathrm{Iw}_H(g)^{w\lambda+\rho}c(\lambda,w), \ \ \ c(\lambda,w)=\prod_{ \alpha \in \Phi_w}\frac{\zeta_X\left(\left\langle\lambda, \alpha^\vee\right\rangle\right)}{\zeta_X\left(1+\left\langle\lambda, \alpha^\vee\right\rangle\right)},
\end{equation}
where $\Phi_w$ is defined in equation~\eqref{eq.Phiw} and $\zeta_X(s)$ is the zeta function of the smooth projective curve $X$ defined over the finite field $\mathbb{F}_q$.

We let $\mathcal{C}=\{x\in \mathfrak{h}_{\mathbb{R}}: \langle\alpha_i, x\rangle>0,~  i\in I\}$ (resp. $\mathcal{C}^*=\{\lambda\in\mathfrak{h}_{\mathbb{R}}^*: \langle\lambda,\alpha_i^\vee\rangle>0, ~ i\in I\}$),
and let $\mathfrak{C} = \mathrm{int}\left(\bigcup_{w \in W} \, w \bar{\mathcal C}\right)$ (resp. $\mathfrak{C}^* = \mathrm{int}\left(\bigcup_{w \in W} \, w \bar{\mathcal C^{\ast}} \right ) $ ) where $\bar{\mathcal{C}}$ (resp. $\bar{\mathcal{C}^{\ast}}$) denotes the closure.
For $h=(h_v)_{v \in \mathcal V} \in H_{\mathbb A}$ with $h_v= \prod_{i \in I} h_{i,v}(t_{i,v})$,  we set $t_i=(t_{i,v})_{v\in\mathcal{V}} \in \mathbb A^\times$ for $i \in I$ as before and write $h= \prod_{i \in I} h_i(t_i)$ by abusing notation slightly.
Using the map $(t_1,\dots, t_n)\mapsto \prod_{i \in I} h_{\alpha_i}(t_i)$, we may identify $(\mathbb{A}^\times)^n$ with $H_{\mathbb{A}}$. In particular, $H_{\mathbb A}$ has the Haar measure. 

We define $H_\mathcal{C}$ (resp.  $H_\mathfrak{C}$) to be
\begin{equation}\label{eq.AdelicTits}
\left\{\prod_{i\in I} h_i(t_i)\in H_{\mathbb{A}}:\sum_{i\in I}\log|t_i| \, \alpha_i^\vee \in \mathcal C \ (\text{\rm resp. } \mathfrak C) \right \}. 
\end{equation}
We note from \eqref{eq.hlambdalocal} that 
\begin{equation} \label{eqhl} H_\mathcal{C}=\left\{h\in H_{\mathbb{A}}:h^{\alpha_i}>1 , i \in I\right \}.\end{equation} 
More generally, for any $\mathcal{K}\subset\mathfrak{C}$, we define 
\[H_{\mathcal{K}}=\left\{\prod_{i \in I} h_i(t_i)\in H_{\mathbb{A}}:\sum_{i \in I}\log|t_i|\, \alpha_i^\vee \in \mathcal K \right \}.\]

The following is a function-field analogue of  \cite[Theorem~3.5]{CGLLM}. Since the proof is similar, we omit it. 
\begin{lemma}\label{lem:absuniMh}
If $\lambda\in\mathfrak{h}_{\mathbb{C}}^{\ast}$ satisfies $\mathrm{Re}(\lambda)\in\mathcal{C}^{\ast}$, then, for any $M>0$ and any compact $\mathcal{K}\subset\mathfrak{C}$, the sum $\sum_{w\in W}M^{\ell(w)}h^{w\lambda}$ converges absolutely and uniformly for $h$ in $H_{\mathcal{K}}$.  
\end{lemma}

Now we state the main result of this section.
\begin{proposition} \label{pro-conv-const}
If $\lambda\in\mathfrak{h}_{\mathbb{C}}^*$ satisfies $\mathrm{Re}(\lambda - \rho) \in \mathcal{C}^*$, 
then $E^\sharp_\lambda(g)$ converges absolutely for $g\in U_{\mathbb{A}}H_\mathfrak{C}\mathbb{K}$. 
Moreover, for any compact $\mathcal{K}\subset\mathfrak{C}$, the convergence is uniform for $\mathrm{Iw}_H(g)$ in $H_{\mathcal{K}}$.
\end{proposition}

\begin{proof}
For a smooth projective curve $X$ defined over $\mathbb{F}_q$, one has $\lim_{\mathrm{Re}(s)\to\infty}\frac{\zeta_X(s)}{\zeta_X(s+1)}=1$.
This follows immediately from the well-known fact that $\zeta_X(s)$ can be written as a  product $\prod_{n=1}^\infty \left ( 1 - \frac 1 {q^{ns}} \right )^{-a_n}$, where $a_n$ is the number of closed points of degree $n$.
Alternatively, this follows directly from the well-known Weil conjecture 
\[
\zeta_X(s) = \frac{ \prod^{2g}_{i=1} (1-\beta_i q^{-s})}{(1-q^{-s})(1-q^{1-s})},
\]
where $g$ is the genus of $X$, and $\beta_i$, $i=1,\ldots, 2g$, are algebraic integers with absolute value $\sqrt{q}$.

We may assume that $\lambda \in \mathfrak h_{\mathbb R}^*$. 
Fix a constant $M >1$, and let $S \ge 1$ be a constant such that $\left|\zeta_X(s)/\zeta_X(s+1)\right|\leq M$ for all $\mathrm{Re}(s)>S$. 
Since, for all $i\in I$, we have $\langle\lambda, \alpha_i^\vee\rangle>1$, we deduce that $\langle\lambda,\alpha^\vee\rangle>S$ for all but finitely many positive roots $\alpha$.  
Therefore, for all $w\in W$, the function $c(\lambda,w)$ defined in equation~\eqref{eee} is bounded by $M^{\ell(w)}$ times a constant independent of $w$, since the cardinality of $\Phi_w$ is equal to $\ell(w)$. 
Now the assertion of the proposition follows from Lemma \ref{lem:absuniMh}.
\end{proof}

As a corollary, we obtain almost everywhere convergence of $E_\lambda(g)$ by applying Tonelli's theorem.
\begin{corollary} \label{cor:almost}
If $\lambda\in\mathfrak{h}_{\mathbb{C}}^*$ satisfies $Re(\lambda -\rho) \in \mathcal{C}^*$, then, for any compact $\mathcal{K}\subset\mathfrak{C}$, there exists a measure zero subset $S_0$ of $(\Gamma \cap U_{\mathbb A}) \backslash U_{\mathbb{A}}H_{\mathcal{K}}$ such that $E_\lambda(g)$ converges absolutely for $g\in U_{\mathbb{A}}H_{\mathcal{K}}\mathbb{K}$ off the set $S_0\mathbb{K}$.
\end{corollary}

\section{Everywhere convergence} \label{Conv}

In this section, we establish everywhere convergence of $E_\lambda$.

\medskip

We start with a lemma which is crucial for the proof of the main theorem.   

\begin{lemma} \label{lem-pm'}
If $h \in H_{\bar {\mathfrak C}}$ and $U_\mathbb{K}:= \mathbb \mathbb{U}_\mathbb{A}\cap \mathbb{K}$, then the image of $hU_\mathbb{K}h^{-1}$ in  $\Gamma\cap U_{\mathbb{A}}\backslash U_{\mathbb{A}}$ has positive measure.
\end{lemma}

\begin{proof}
Since a fundamental domain of $\mathbb A/F$ can be chosen inside $\prod_{v \in \mathcal V} \mathcal O_v$, we see that $U_\mathbb{K}$  has positive measure in $\Gamma\cap U_{\mathbb{A}}\backslash U_{\mathbb{A}}$. 
By assumption, we have that 
\[
h= w^{-1} h_1 w
\]
for some $w \in W$ and $h_1 \in H_{\bar{\mathcal C}}$.  
Write $U^{w}=U_{\mathbb{A}}\cap w^{-1}U_{\mathbb{A}}w$ and $U_{w}=U_{\mathbb{A}}\cap w^{-1}U^-_{\mathbb{A}}w$, where $U^-_{\mathbb A}$ is the opposite unipotent group. 
Then we may decompose the space $U_{\mathbb{A}}$ as a product:
\begin{equation}\label{eq.decomposition'}
U_{\mathbb{A}}=U^{w}U_{w}.
\end{equation}
Note that $U_{w}\cong\mathbb{A}^{\ell(w)}$. 
Likewise, 
\[
U_\mathbb{K} = U^{w}_\mathbb{K} U_{w, \mathbb{K}},
\]
where $U^{w}_\mathbb{K}  := U^{w} \cap \mathbb{K}$ and $U_{w, \mathbb{K}}: = U_{w}\cap \mathbb{K}$. 

Recall that we fixed  a regular dominant integral weight $\Lambda \in \mathfrak h_{\mathbb C}^*$ and that $(V,\pi)$ is the irreducible highest weight representation of $\mathfrak g_{\mathbb C}$ with highest weight $\Lambda$. 
The group $G_v$ is defined as a subgroup of $\mathrm{Aut}(V_{F_v})$ for each $v\in\mathcal{V}$, and the adelic group $G_{\mathbb{A}}$ is a restricted product of $G_v$.
For each $v\in\mathcal{V}$, let $\mathfrak X_v$ denote the coherently ordered basis for $V_{F_v}$ induced from $\mathfrak X$ for $V_{\mathbb Z}$, and consider $x_{v,\mu} \in \mathfrak X_v$ with weight $\mu$. 
Let $U^{w}_v$ be the local component of $U^{w}$ at the place $v$.
For $u_v \in U^{w}_v$, we have that $w u_v w^{-1}\in U_v$, thus we can write 
\[
wu_v x_{v, \mu} =  (w u_v w^{-1}) wx_{v,\mu} = \sum_{ \nu \succeq w\mu} \tilde x_{v,\nu}, \qquad \text{ or } \qquad u_v x_{v, \mu} = \sum_{ \nu \succeq w\mu} w^{-1} \tilde x_{v,\nu}, 
\]
where $\tilde x_{v,\nu}$ are of weight $\nu$ and $\succeq$ is the usual partial order, i.e., $\nu \succeq w\mu$ if and only if $\nu -w\mu$ is a non-negative integral sum of simple roots $\alpha_i$ ($i \in I$).
For an integral weight $\lambda$ of $\mathfrak g_{\mathbb C}$ and $h_v = \prod_{i \in I} h_i(t_{i,v}) \in H_v$, $t_{i,v} \in F_v^\times$, define  
\[ 
h_v^\lambda  = \prod_{i \in I} t_{i,v}^{\langle \lambda , \alpha_i^\vee \rangle}  \in F_v^\times. 
\]
Write $h=(h_v)_{v\in\mathcal{V}} = (w^{-1} h_{1,v} w)_{v\in \mathcal{V}} \in H_{\overline {\mathfrak C}}$. 
Then we have 
\begin{align*}
 h_v^{-1} u_v h_v x_{v,\mu} & =  h_v^{\mu} \cdot h_v^{-1} u_v x_{v,\mu}  = h_v^{\mu}\cdot w^{-1}h_{1,v}^{-1} w u_v x_{v,\mu} \\
 & = h_v^{\mu}\cdot  w^{-1}h_{1,v}^{-1} \sum_{\nu\succeq w\mu } \tilde x_{v, \nu}   =  \sum_{\nu \succeq w\mu}h_v^{\mu-w^{-1}\nu} \cdot w^{-1} \tilde x_{v,\nu}  \\
 & = \sum_{\nu\succeq w\mu} h_{1,v}^{w\mu-\nu}\cdot w^{-1} \tilde x_{v,\nu}
\end{align*}  at each place $v \in \mathcal V$.
 Considering $u_v$ as an infinite matrix, we see that the each entry in the block corresponding to $(\mu, \nu)$ has been multiplied by $h_{1,v}^{w\mu-\nu}$ after conjugation by $h_v^{-1}$. 
Globally, it follows from \eqref{eqhl} that $h_1^{w\mu-\nu}\leq 1$. 
In particular, taking $u=(u_v)_{v\in\mathcal{V}}\in U^w_\mathbb{K}$,  we see that 
\[
h^{-1} U^w_\mathbb{K} h \subset U^w_\mathbb{K}.
\]
It follows that $U^w_\mathbb{K} \subset hU^w_\mathbb{K} h^{-1}$ hence
\[
U^w_\mathbb{K} \cdot h U_{w,\mathbb{K}} h^{-1} \subset hU_\mathbb{K} h^{-1}.
\]
Since $U_w\cong \mathbb{A}^{\ell(w)}$ is locally compact, its open compact subgroups $hU_{w,\mathbb{K}} h^{-1}$ and $U_{w,\mathbb{K}}$ are commensurable. 
As $U_\mathbb{K} = U^w_\mathbb{K} U_{w, \mathbb{K}}$ has positive measure in $\Gamma\cap U_\mathbb{A}$,  so do $U^w_\mathbb{K} \cdot h U_{w,\mathbb{K}} h^{-1} $ 
and $hU_\mathbb{K} h^{-1}$.
\end{proof}

Now we prove the main theorem.
\begin{theorem}[Theorem \ref{conv.thm.intro}]\label{convergence.theorem}
If $\lambda\in\mathfrak{h}^{\ast}_{\mathbb{C}}$ satisfies $Re(\lambda - \rho) \in \mathcal{C}^*$, then the series $E_\lambda(g)$ converges absolutely for $g\in U_{\mathbb{A}}H_\mathfrak{C}\mathbb{K}$.
\end{theorem}

\begin{proof}
We may consider $E_\lambda$ as a function on $\left(\Gamma \cap U_{\mathbb A} \backslash U_{\mathbb A}\right) \times H_{\mathbb A}$.
Assume that $E_\lambda (u h) = \infty$ for some $u\in U_{\mathbb A}$ and $h \in H_{\mathfrak{C}}$. 
By Lemma \ref{lem-pm'}, $U':= hU_\mathbb{K}h^{-1}$ has positive measure in $\Gamma\cap U_\mathbb{A}\backslash U_\mathbb{A}$. 
Then we have 
\[
 E_\lambda(u u'  h) = E_\lambda(u h (h^{-1}u' h)) = E_\lambda (u h) = \infty
\]
for any $u' \in U'$.  
Since finite dimensional unipotent groups are unimodular and $\Gamma\cap U_\mathbb{A}\backslash U_\mathbb{A}$ admits the projective limit measure, $uU'$ and $U'$ have the same measure in  $\Gamma\cap U_\mathbb{A}\backslash U_\mathbb{A}$, which is positive.
This is a contradiction to Corollary~\ref{cor:almost} and completes the proof.
\end{proof}

\end{document}